\documentclass[12pt,a4paper]{amsart}%

\usepackage{amsfonts,amsmath,amssymb}
\usepackage{amssymb,amsthm,amsxtra}
\usepackage[usenames]{color}
\usepackage{amscd}
\usepackage{amsthm}
\usepackage{amsfonts}
\usepackage{amssymb}
\usepackage{amsmath}
\usepackage{graphicx}
\usepackage{hyperref}
\usepackage{enumerate}%
\usepackage[all]{xy}
\usepackage[usenames,dvipsnames]{xcolor}
\usepackage{mathrsfs}
\usepackage[capitalize]{cleveref}
\usepackage{cleveref}
 \newtheorem*{corollary*}{Corollary}
 \newtheorem*{construction*}{Construction}
 \newtheorem*{definition*}{Definition}
 \newtheorem*{notation*}{Notation}
 \newtheorem*{lemma*}{Lemma}
 \newtheorem*{theorem*}{Theorem}
 \newtheorem*{remark*}{Remark}
 \newtheorem*{example*}{Example}
 \newtheorem*{conjecture*}{Conjecture}
 \newtheorem*{condition*}{Condition}
 \newtheorem*{result*}{Result}
 \newtheorem*{property*}{Property}

 \newtheorem*{cor*}{Corollary}
 \newtheorem*{const*}{Construction}
 \newtheorem*{defn*}{Definition}
 \newtheorem*{notn*}{Notation}
 \newtheorem*{lem*}{Lemma}
 \newtheorem*{thm*}{Theorem}
 \newtheorem*{rem*}{Remark}
 \newtheorem*{exm*}{Example}
 \newtheorem*{conj*}{Conjecture}

 \newtheorem{lemma}{Lemma}[section]
 
 \newtheorem{remark}[lemma]{Remark}
 
 \newtheorem{theorem}[lemma]{Theorem}
 
 \newtheorem{notation}[lemma]{Notation}

 \newtheorem{thm}[lemma]{Theorem}
 \newtheorem{prop}[lemma]{Proposition}
 \newtheorem{lem}[lemma]{Lemma}
 \newtheorem{defn}[lemma]{Definition}
 \newtheorem{notn}[lemma]{Notation}
 \newtheorem{cor}[lemma]{Corollary}

 \newtheorem{introtheorem}{Theorem}

 \crefname{introtheorem}{theorem}{theorems}
 \Crefname{introtheorem}{Theorem}{Theorems}
  \newtheorem{introthm}[introtheorem]{Theorem}
   \crefname{introthm}{theorem}{theorems}
 \Crefname{introthm}{Theorem}{Theorems}

  \crefname{introcorollary}{corollary}{corollaries}
 \Crefname{introcorollary}{Corollary}{Corollaries}

   \crefname{introcor}{corollary}{corollaries}
 \Crefname{introcor}{Corollary}{Corollaries}
 
   \crefname{introconjecture}{conjectures}{conjectures}
 \Crefname{introconjecture}{Conjecture}{Conjectures}
 
    \crefname{introconj}{conjectures}{conjectures}
 \Crefname{introconj}{Conjecture}{Conjectures}

     \crefname{introlem}{lemma}{lemmas}
 \Crefname{introlem}{Lemma}{Lemmas}
 
 \crefname{introremark}{remark}{remarks}
 \Crefname{introremark}{Remark}{Remarks}
 
  \crefname{introrem}{remark}{remarks}
 \Crefname{introrem}{Remark}{Remarks}

   \crefname{introprop}{Proposition}{Propositions}
 \Crefname{introprop}{Proposition}{Propositions}

   \crefname{introdefn}{definition}{definitions}
 \Crefname{introdefn}{Definition}{Definitions}
 
   \crefname{intronotn}{notation}{notations}
 \Crefname{intronotn}{Notation}{Notations}
 
   \crefname{introtask}{task}{tasks}
 \Crefname{introtask}{Task}{Tasks}
 
  \crefname{introprob}{problem}{problems}
 \Crefname{introprob}{Problem}{Problems}
 
   \crefname{introquestion}{question}{questions}
 \Crefname{introquestion}{Question}{Questions}

   \crefname{introexm}{example}{example}
 \Crefname{introquestion}{Example}{Example}
 \crefname{theorem}{theorem}{theorems}
 \Crefname{theorem}{Theorem}{Theorems}
  \crefname{thm}{theorem}{theorems}
 \Crefname{thm}{Theorem}{Theorems}
  \crefname{corollary}{Corollary}{Corollaries}
 \Crefname{corollary}{Corollary}{Corollaries}
   \crefname{cor}{Corollary}{Corollaries}
 \Crefname{cor}{Corollary}{Corollaries}
   \crefname{conjecture}{conjectures}{conjectures}
 \Crefname{conjecture}{Conjecture}{Conjectures}

    \crefname{conj}{conjectures}{conjectures}
 \Crefname{conj}{Conjecture}{Conjectures}

     \crefname{lem}{lemma}{lemmas}
 \Crefname{lem}{Lemma}{Lemmas}

      \crefname{lemma}{Lemma}{Lemmas}
 \Crefname{lemma}{Lemma}{Lemmas}

 \crefname{remark}{remark}{remarks}
 \Crefname{remark}{Remark}{Remarks}

  \crefname{rem}{remark}{remarks}
 \Crefname{rem}{Remark}{Remarks}

   \crefname{rem}{remark}{remarks}
 \Crefname{rem}{Remark}{Remarks}

   \crefname{proposition}{Proposition}{Proposition}
 \Crefname{proposition}{Proposition}{Proposition}

    \crefname{prop}{Proposition}{Propositions}
 \Crefname{prop}{Proposition}{Propositions}

   \crefname{defn}{definition}{definitions}
 \Crefname{defn}{Definition}{Definitions}

   \crefname{notn}{notation}{notations}
 \Crefname{notn}{Notation}{Notations}

   \crefname{task}{task}{tasks}
 \Crefname{task}{Task}{Tasks}

  \crefname{prob}{problem}{problems}
 \Crefname{prob}{Problem}{Problems}

   \crefname{question}{question}{questions}
 \Crefname{question}{Question}{Questions}

\newcommand{\Ind}{\operatorname{Ind}}
\newcommand{\ind}{\operatorname{ind}}

\newcommand{\ch}{\operatorname{char}}

\newcommand{\GL}{\operatorname{GL}}

\newcommand{\sll}{{\mathfrak{sl}}}

\newcommand{\WF}{\operatorname{WF}}

\newcommand{\irr}{\operatorname{irr}}

\newcommand{\bN}{\mathbb{N}}

\newcommand{\bfG}{\mathbf{G}}

\newcommand{\bfM}{\mathbf{M}}

\newcommand{\bfP}{\mathbf{P}}

\newcommand{\R}{\mathbb{R}}

\providecommand{\fg}{\mathfrak{g}}

\providecommand{\fm}{\mathfrak{m}}

\providecommand{\cL}{\mathcal{L}}

\providecommand{\sub}{\subset}

\newcommand{\Dima}[1]{{{#1}}}
\newcommand{\DimaA}[1]{{{#1}}}
\newcommand{\DimaB}[1]{{#1}}
\newcommand{\DimaC}[1]{{{#1}}}
\newcommand{\Rami}[1]{{{#1}}}
\newcommand{\RamiA}[1]{{{#1}}}
\newcommand{\RamiB}[1]{{{#1}}}
\newcommand{\RamiC}[1]{{{#1}}}
\newcommand{\RamiD}[1]{{{#1}}}
\newcommand{\Eitan}[1]{{{#1}}}

\newcommand{\NO}{\mathcal N_{o}(G)}
\newcommand{\NOBM}{\mathcal I_o(G)}

\begin{document}

\title{Irreducibility of wave-front sets for depth zero cuspidal representations}

\author{Avraham Aizenbud}
\address{Avraham Aizenbud, Faculty of Mathematics and Computer Science, Weizmann
Institute of Science, POB 26, Rehovot 76100, Israel }
\email{aizenr@gmail.com}
\urladdr{http://www.aizenbud.org}

\author{Dmitry Gourevitch}
\address{Dmitry Gourevitch, Faculty of Mathematics and Computer Science, Weizmann
Institute of Science, POB 26, Rehovot 76100, Israel }
\email{dimagur@weizmann.ac.il}
\urladdr{http://www.wisdom.weizmann.ac.il/~dimagur}

\author{Eitan Sayag}
\address{Eitan Sayag,
 Department of Mathematics,
Ben Gurion University of the Negev,
P.O.B. 653,
Be'er Sheva 84105,
ISRAEL}
 \email{eitan.sayag@gmail.com}

%
%
%
%
%
%
%
%
%
\date{\today}

\maketitle
\begin{abstract}
\RamiB{We show that the results of \cite{BM, DebPar,Oka,Lus,AAu,Tay} imply a \DimaC{positive answer to the question of Moeglin-Waldspurger on wave-front sets in the case of depth zero cuspidal representations.} 
Namely, we deduce that for large enough residue characteristic, the Zariski closure of the wave-front set of any depth zero irreducible cuspidal representation of any reductive group over a non-Archimedean local field is \RamiC{an irreducible variety}.

In more details, we use \cite{BM, DebPar,Oka} to reduce the statement to an analogous statement for finite groups of Lie type, which is proven in  \cite{Lus,AAu,Tay}.}

%

\end{abstract}

\tableofcontents

\section{Introduction}

In this paper we prove the following theorem.
\begin{introthm}\label{thm:main}
For any $n\in \bN$ there exists $T\in \bN$ such that for any

\begin{itemize}
        \item prime $p>T$,
        \item local field $F$ of residue characteristic $p$ such that $val_F(p)<n$,
        \item reductive group $\bf G$ defined over $F$ such that $\dim \bfG<n$,
        \item cuspidal irreducible representation $\pi$ of depth zero of $\bfG(F)$,
\end{itemize}
 the Zariski closure of the wave-front set $\WF(\pi)$ in $\fg^*$ is irreducible, where $\fg$ denotes the Lie algebra of $\bfG$.

\end{introthm}

In \S \ref{sec:form} we formulate a more explicit version of this theorem.

\subsection{Idea of the proof}

The natural analogue of Theorem \ref{thm:main} for finite groups of Lie type was proven in \cite{Lus,AAu,Tay}. Barbasch and Moy \cite{BM} \RamiA{provide} a method to describe the wave-front set of a depth zero representation of $\bfG(F)$ in terms of certain \Dima{representations} of certain finite groups of Lie type. In general, the description is quite complicated, but for cuspidal representations we \RamiA{make} this description very explicit (see Corollary  \ref{cor:cusp.BM} below). This explicit description together with the result of \cite{Lus,AAu,Tay} implies Theorem \ref{thm:main}.

\subsection{Related results}
The irreducibility of the wave-front set of irreducible representations of finite groups of Lie type was conjectured (in a different language) in \cite{Kaw}. This conjecture was proven in \cite{Lus,AAu,Tay}.

For irreducible representations of p-adic groups \Rami{the irreducibility of the  wave-front set was suggested} in \cite{MW87} and proven for some cases, including all irreducible representations of $\GL_n$\Rami{, see} \cite[Chapter II]{MW87}. \Rami{In \cite{Wal18, Wal20}, the irreducibility of the  wave-front set} was proven for \Rami{many cases including} anti-tempered and tempered
unipotent representations of groups in the inner class of the split form of $\mathrm{SO}(2n+1)$. Recently, it was proven for irreducible Iwahori-spherical depth zero representations \Eitan{with a {\it real infinitesimal character}}, \RamiA{see} \cite[Theorem 1.3.1]{CMBO}.

\DimaC{Very recently, examples of irreducible representations of p-adic groups with reducible wave-front sets were given in \cite{Tsai}. 

Theorem \ref{thm:main} is independently proven in \cite[\S 2.5]{CMBO2}.
}
\RamiA{
\subsection{Acknowledgement}
We thank Dan Ciubotaru and Emile Okada for fruitful communications.
\RamiD{We thank Jeffrey Adler for answering our questions regarding the Bruhat-Tits building.
We also thank the referee for his useful remarks.}}

\section{Formulation of the main result}\label{sec:form}
Throughout the paper we fix a reductive algebraic group $\bfG$ defined over a  local non-Archimedean \RamiA{field} $F$.
\begin{notation}We denote:
\begin{itemize}
\item  $k$ -- the residue field of $F$,
\item  $\fg$  -- the  Lie algebra of $\bfG$,
\item $\bfG'$  --  the derived group of $\bfG$,
\item $\fg'$  --  the  Lie algebra of $\bfG'$,
\item  $G=\bfG(F)$,
\item  $BT(G)$  --  the Bruhat-Tits building of $\bfG$,
\item $\irr(G)$ -- the set of (isomorphism classes of) irreducible smooth representations of $G$.
\end{itemize}
\RamiA{For} any $x\in BT(G)$ and $r\in \R$ we \RamiA{further denote:}
\begin{itemize}
\item $G_{x,r}$ and  $G_{x,r^+}$ -- the Moy-Prasad subgroups (See \Rami{\cite[\S2.6]{MP94} \RamiA{where} \Dima{they are denoted by $\mathcal P_{x,r}, \mathcal P_{x,r^+}$}}).
\item $M_x:=G_{x,0}/G_{x,0^+}$.
\item  $\bfM_{x}$ -- \RamiA{the} natural reductive group defined over $k$ s.t.
$M_x\cong \bfM_{x}(k)$ (See \Rami{\cite[\S3.2]{MP94}})
\item $Q_{x}$ -- the normalizer of $G_{x,0}$ inside $G$.
\item $G_{r}:=\bigcup_{x\in BT(G)}  G_{x,r}$
\item $G_{r^+}:=\bigcup_{x\in BT(G)}  G_{x,r^+}$
\item $\fg_{x,r}, \fg_{x,r^+}, \fm_x, \fg_{r}, \fg_{r^+}$ --  the Lie algebra versions of the above \Rami{(See \cite[\S3]{MP94})}.

\end{itemize}
\end{notation}

\begin{defn}\label{def:acc}
        We say that
        \RamiA{the pair $(F,\bfG)$}
        is \emph{acceptable}, if the following conditions are satisfied:

                \begin{enumerate}
                \item \label{it:HCH} The pair $(F,\bfG)$ satisfies the Hales-Moy-Prasad conjecture \Rami{for depth $0$ representations}, {\it i.e.} for any depth $0$ representation  $\rho\in \irr(G)$, the Harish-Chandra-Howe character expansion for $\rho$ is valid on the set $G_{0^+}$ of topologically unipotent elements in $G$.

                \item \label{it:exp} The series defining the exponential map $\fg'\to \bfG$ given by the adjoint representation converge on $\fg_{0+} \cap \fg'$.

                \item \label{it:Cox} For any $x\in BT(G)$, we have  $\ch  k>3(h_x -1)$, where $h_x$ is the Coxeter number of ${\bf M}_x$. Note that in particular this implies that $p$ is a good prime for ${\bf M}_x$.
        \end{enumerate}
\end{defn}

\begin{prop}
For any $n\in \bN$ there exists $T\in \bN$ such that for any

\begin{itemize}
        \item prime $p>T$
        \item local field $F$ of residue characteristic $p$ such that $val_F(p)<n$ \item reductive group $\bf G$ defined over $F$ such that $\dim \bfG<n$,
\end{itemize}
the pair $(F,\bfG)$ is acceptable.
\end{prop}
\begin{proof}
In order to satisfy condition \eqref{it:Cox} we take $T>3n$.
In order to satisfy condition \eqref{it:exp} we take $T>n^2$. It suffices by \cite[Lemma 3.2]{BM}. Finally, one can choose $T$ such that \eqref{it:HCH} is satisfied by \cite[\S 2.2, \S3.4, Theorem 3.5.2]{DebHom} and condition \eqref{it:exp}.
\end{proof}

\begin{defn}
For $\pi\in \irr(G)$ denote by $\WF(\pi)$ the wave-front set of the
character of $\pi$ over the point $1\in G$. It also equals
the union of the closures of all the orbits \Eitan{ $O \subset \mathfrak{g}^*(F)$} that appear with non-zero coefficients in the Harish-Chandra-Howe expansion.
\end{defn}

The following is a more explicit version of Theorem \ref{thm:main}.

\begin{thm}\label{thm:main.exp}
Assume that $(F, \bfG)$ is an acceptable pair.
Let $\pi$ be a cuspidal irreducible representation of depth zero of $\bfG(F)$. Then the Zariski closure of the wave-front set $\WF(\pi)$ in $\fg^*$ is irreducible.
\end{thm}

From now till the end of the paper we will assume that $(F, \bfG)$ is an acceptable pair. The rest of the paper is dedicated to the proof of this theorem.

%
%

\section{Wave-front sets and generalized Gelfand-Graev models for finite groups of Lie type}

Let $k$ be a finite field, $\bf M$ be a reductive group defined over $k$, and $M:={\bf M}(k)$ be its group of $k$-points.
We assume that  $\ch  k>3(h -1)$, where $h$ is the Coxeter number of $\bfM$. In particular, this implies that $\ch k$ is a good prime for $\bf M$. Let $\fm$ denote the  Lie algebra of $\bf M$.



\RamiA{
\begin{defn}$ $
        \begin{itemize}
                \item
For every nilpotent element $N\in \fm(k)$, let  $\Gamma_{N}$ denote the \emph{generalized Gelfand-Graev model} attached to $N$, as in \cite[\S 2.2]{BM}. Since the isomorphism class of $\Gamma_N$ only depends on the orbit of $N$ \Dima{under the adjoint action of $M$}, we will also  use the notation $\Gamma_O$
for every nilpotent orbit $O\subset \fm(k)$.
\item Let $\sigma$ be a finite-dimensional representation of $M$.
Let $GG(\sigma)$ denote the union of all nilpotent $M$-orbits $O\sub \fm(k)$ satisfying $\langle \tau, \Gamma_O\rangle\neq 0,$ where $\langle \tau, \Gamma_O\rangle$ denotes the intertwining number.
\item
Let $\WF(\sigma)$ denote the Zariski closure of $\bfM\cdot GG(\sigma)$ in $\fm$.
\end{itemize}
\end{defn}
}

\begin{thm}[{\cite[Theorem 14.10]{Tay}}]\label{thm:lus}
If $\ch k$ is a good prime for $\bf M$ then for every irreducible representation $\sigma$ of $M$, the algebraic variety $\WF(\sigma)$ is irreducible.
\end{thm}

\section{The results of Barbasch-Moy}
In this section we describe the results of \cite{BM}\RamiB{, as refined in \cite{DebPar,Oka},} on the relation between wave-front sets of depth zero representations of $G$ and of representations of the \RamiA{finite} groups $M_x$ for various $x\in BT(G)$.

The results in \cite{BM} require certain assumptions, described in \cite[4.4]{BM}.
Assumptions (2) and (3) in \cite[4.4]{BM} coincide with assumptions \eqref{it:exp} and \eqref{it:Cox} in Definition \ref{def:acc}. Assumption (1) in \cite[4.4]{BM} can be replaced by assumption \eqref{it:HCH} of Definition \ref{def:acc}. Indeed, this assumption \RamiA{is} only used in \cite{BM} in order to deduce the statement of assumption \eqref{it:HCH} of Definition \ref{def:acc}.
Therefore all the results of \cite{BM} are valid for the acceptable pair $(F,\bfG)$.

\begin{defn}
Let
\begin{enumerate}[(i)]
\item
$\NO$ denote the set of nilpotent $G$-orbits in $\fg(F)$.
\item $\NOBM=\{(x,O)\, \vert  x\in BT(G),\, O \text{ is a nilpotent \RamiB{$M_x-$}orbit in }\fm_x\}.$
\item $F^{un}$ be the  unramified closure of $F$.

\item
We define a pre-order on $\NO$ in the following way:
$$\Omega\geq \Omega'\quad  \text{iff} \quad \overline{G(F^{un})\cdot \Omega}\supset \Omega',$$ where
$\overline{\DimaC{\bfG}(F^{un})\cdot \Omega}$ denotes the closure in  the local topology on $\fg(F^{un})$.
\item 
\RamiA{We define a \RamiB{pre-}order on \RamiB{$\NOBM$} in the following way:
        $$(x',O')\leq(x,O) \text{ iff } x=x' \text{ and }  \overline{\RamiB{\bfM_x \cdot} O}^{Zar}  \supset O',$$
where $\DimaC{\overline{\bfM_x \cdot O}}^{Zar}$ denotes the Zariski closure.}
%
%
\end{enumerate}
\end{defn}
\RamiB{
\begin{thm}[{\cite[Lemma 5.3.3.]{DebPar}}]
        For any $(x,O)\in \NOBM$ there exists \Eitan{a} unique $\Omega \in \NO$ s.t. there exists an $\sll_2$-triple $e,h,f\in \fg_{x,0}$ satisfying:
        \begin{itemize}
         \item $e\in \Omega$
         \item the projections $\bar e, \bar h, \bar f$  form an $\sll_2$-triple in  $\fm_x(k)=\fg_{x,0}/\fg_{x,0^+}$
         \item $\bar e\in O$.
\end{itemize}
\end{thm}
\begin{notn}
        We will denote: $$\cL(x,O):=\Omega$$
\end{notn}

\begin{thm}[{\cite{BM,DebPar}}]\label{thm:cor}$\,$
         The map $\cL:\NOBM \to \NO$ is:
        \begin{enumerate}[(i)]
                \item surjective, cf. \cite[Theorem 5.6.1]{DebPar},
                \item pre-order preserving, cf.  \cite[Proposition 3.16]{BM}.
        \end{enumerate}
\end{thm}

}

\begin{notation}
For $x\in BT(G)$ and an \RamiB{$M_x$}-stable subset $\Xi\subset \fm_x$ denote \RamiB{$$\cL_x(\Xi)=\bigcup_{O\in \Xi/M_x} \cL(x,O).$$}
\end{notation}

\begin{defn}
For $\pi\in \irr(G)$ and $x\in BT(G)$, define $\pi_x:=\pi^{G_{x,0^+}}$ considered as a representation of $M_x=G_{x,0}/G_{x,0^+}$.
\end{defn}


\begin{thm}[\RamiB{{\cite[Corollary 1.2]{Oka}}\DimaA{, cf. \cite[\S 5]{BM}}}]\label{thm:BMCor}
Let $\pi$ be a representation of $G$ of depth zero. Then we have
$$\overline{\WF(\pi)}^{Zar}=\overline{\bigcup_{x\in BT(G)}\RamiB{\cL_x(\WF(\pi_x)(k))}}^{Zar},$$
\Dima{where by $\overline{\WF(\pi)}^{Zar}$ we mean the closure in the Zariski topology.}
\end{thm}

\section{Proof of Theorem \ref{thm:main}}

\RamiA{We will give an explicit version of the results of \cite{BM} for cuspidal representations of depth 0. For this we f}irst recall a construction from \RamiA{\cite{MP96}}  that exhausts all  depth zero irreducible cuspidal representations of $G$:

\begin{theorem}[{\cite[Propositions 6.6 and 6.8]{MP96}}]\label{thm:cusp.dep.0}
Let $x\in BT(G)$ s.t. $G_{x,0}$ is a maximal parahoric subgroup of $G$, and let $\tau_0\in \irr(Q_x/G_{x,0^+})$ such that $\tau_0|_{M_x}$ is a cuspidal representation. Let $\tau$ be the lift of $\tau_0$ to  $Q_x$. Then  $\pi:=\ind_{Q_x}^G\tau$ is a depth zero cuspidal irreducible representation of $G$. Moreover, any depth zero cuspidal irreducible representation of $G$ can be obtained in this way.
\end{theorem}


\begin{prop}\label{prop:allOr0}
        Let $x\in BT(G),\tau_0 \in  \irr(Q_x/G_{x,0^+}),$ its lift  $\tau \in \irr(Q_x)$ and $\pi=\ind_{Q_{x}}^G\tau\in \irr(G)$ be as in Theorem \ref{thm:cusp.dep.0},
        and let $y\in BT(G)$. Then
        \begin{enumerate}[(i)]
                \item \label{it:0}If $\pi_y\neq 0$ then there exists $g\in G$ such that $gG_{x,0}g^{-1}=G_{y,0}$.
                \item \label{it:Norm}$\pi_x\simeq (\tau_0)|_{M_x}.$
        \end{enumerate}
\end{prop}

For the proof we will need the following lemma.

\begin{lem}
        \label{lem:BT}
Let $x,y\in BT(G)$, and let $F_x$ and $F_y$ denote the minimal faces that include them. \RamiD{Assume that $F_x$ is a face of minimal dimension.}
If $F_x\neq F_y$ then the image of  $G_{x,0}\cap G_{y,0^+}$ in $G_{x,0}\slash G_{x,0^+}$ includes the unipotent radical of a proper parabolic subgroup of ${\bf M}_x(k)$.
\end{lem}
\RamiD{
This Lemma is well known, however, we could not find an exact reference. Thus,  for completeness,  we deduce it here from  what we found in the literature.     
\begin{proof}
        By passing to an unramified extension of $F$ we can and will assume that $\bfG$ is quasi-split. Passing to a cover of $\bfG$ we may further assume that $\bfG$ is a product of a simply connected group and a torus. Since the torus component is inessential for our claim, we may ignore it and assume that $\bfG$ is simply connected. We recall that in this case the parahoric subgroup $G_{x,0}$ is just the stabilizer of  $x$.
        
        By
        \cite[Corollary 3.24]{Rab}
         there is a parabolic $\bfP\subset {\bf M}_x$ such that $G_{x,0}\cap G_{y,0}=\bfP(k)$ and  
         $G_{x,0}\cap G_{y,0+}=\mathbf{U}(k)$, 
         where $\mathbf U\subset \bfP$ is the unipotent radical.\footnote{\cite{Rab} consider only the case that $G$ is split,
                 but the proof of this statement (as well as the other statements from \cite{Rab} that we use) does not use require this assumption.}      
         It is left to show that $\bfP$ is a proper parabolic. 
         
         Choose an apartment $A\subset BT(G)$ that contains the points $x$ and $y$. Connect these points by a segment $I\subset A$. Note that the intersection  $G_{x,0}\cap G_{y,0}$ fixes $I$. So for any $z\in I$ we have $G_{x,0}\cap G_{z,0}\supset G_{x,0}\cap G_{y,0}$. Therefore without loss of generality we may and will assume that $x\in \bar F_y$.
         The statement in this case follows from \cite[Proposition 3.22]{Rab}.
         
\end{proof}
}


\begin{proof}[Proof of Proposition \ref{prop:allOr0}]
We have the following isomorphisms of vector spaces.
\begin{multline*}\pi_y\simeq\bigoplus_{[g]\in Q_x\backslash G \slash G_{y,0^+}} (\Ind^{G_{y,0^+}}_{G_{y,0^+}\cap g^{-1}Q_x g}(\tau|_{gG_{y,0^+}g^{-1}\cap Q_x})^g)^{G_{y,0^+}}\simeq\\
\bigoplus_{[g]\in Q_x\backslash G \slash G_{y,0^+}} (\Ind^{G_{y,0^+}}_{G_{y,0^+}\cap Q_{g^{-1}x}}(\tau|_{G_{gy,0^+}\cap Q_x})^g)^{G_{y,0^+}}\simeq
\bigoplus_{[g]\in Q_x\backslash G \slash G_{y,0^+}} ((\tau|_{G_{gy,0^+}\cap Q_x})^g)^{G_{y,0^+}\cap Q_{g^{-1}x}}\simeq\\
\bigoplus_{[g]\in Q_x\backslash G \slash G_{y,0^+}} \tau^{G_{gy,0^+}\cap Q_x}
\end{multline*}
By Lemma \ref{lem:BT} and the cuspidality of $\tau_0$ we obtain
$$\pi_y\simeq \bigoplus_{[g]\in Q_x\backslash G \slash G_{y,0^+} \text{ s.t. }F_{gy}=F_x} \tau^{G_{x,0^+}}\simeq\bigoplus_{[g]\in Q_x\backslash G \slash G_{y,0^+} \text{ s.t. }F_{gy}=F_x} \tau_0.$$ This proves \eqref{it:0}.
To prove \eqref{it:Norm} we use the following isomorphism of representations of $G_{x,0}$.
 \begin{multline}\label{=pix}
 \pi_x\simeq\bigoplus_{[g]\in Q_x\backslash G \slash G_{x,0}} (\Ind^{G_{x,0}}_{G_{x,0}\cap g^{-1}Q_xg}(\tau|_{gG_{x,0}g^{-1}\cap Q_x})^g)^{G_{x,0^+}}\cong \\
\bigoplus_{[g]\in Q_x\backslash G \slash G_{x,0}} (\Ind^{G_{x,0}}_{G_{x,0}\cap Q_{g^{-1}x}}(\tau|_{G_{gx,0}\cap Q_x})^g)^{G_{x,0^+}}
\end{multline}
For any $g\in G$ we have a vector space isomorphism:
\begin{multline*}(\Ind^{G_{x,0}}_{G_{x,0}\cap Q_{g^{-1}x}}(\tau|_{G_{gx,0}\cap Q_x})^g)^{G_{x,0^+}}\cong
\bigoplus_{[h]\in Q_x\backslash Q_xgG_{x,0} \slash G_{x,0^+}} (\Ind^{G_{x,0^+}}_{G_{x,0^+}\cap Q_{h^{-1}x}}(\tau|_{G_{hx,0^+}\cap Q_{x}})^h)^{G_{x,0^+}}\cong\\
\bigoplus_{[h]\in Q_x\backslash Q_xgG_{x,0} \slash G_{x,0^+}} ((\tau|_{G_{hx,0^+}\cap Q_{x}})^{h})^{G_{x,0^+}\cap Q_{h^{-1}x}}
\cong
\bigoplus_{[h]\in Q_x\backslash Q_xgG_{x,0} \slash G_{x,0^+}} \tau^{G_{hx,0^+}\cap Q_{x}}
\end{multline*}
By Lemma \ref{lem:BT} and the cuspidality of $\tau$, if the space above does not vanish then for some  $h\in Q_{x}gG_{x,0}$ we have  $F_x=F_{hx}$. In other words $ Q_{x}gG_{x,0}$ intersects $Q_{x}$, and thus $g\in Q_x$. To sum up, if $(\Ind^{G_{x,0}}_{G_{x,0}\cap Q_{g^{-1}x}}(\tau|_{G_{gx,0}\cap Q_x})^g)^{G_{x,0^+}}\neq 0$ then $g\in Q_x$. Using \eqref{=pix} we obtain
 $$\pi_x\simeq (\tau|_{G_{x,0}})^{G_{x,0^+}}=(\tau_0)|_{M_x}$$
\end{proof}

\Eitan{
\begin{remark}
Proposition \ref{prop:allOr0} can be deduced from the main result of \cite{Lath}. We thank the referee for informing us on this reference.
\end{remark}

}

 Proposition \ref{prop:allOr0} and Theorem \ref{thm:BMCor} imply the following corollary.
\begin{cor} \label{cor:cusp.BM}
        Let $x\in BT(G),\tau_0 \in  \irr(Q_x/G_{x,0^+}),$ its lift  $\tau \in \irr(Q_x)$ and $\pi=\ind_{Q_{x}}^G\tau\in \irr(G)$ be as in Theorem \ref{thm:cusp.dep.0},
Then $$\overline{\WF(\pi)}^{Zar}=\overline{\RamiB{\cL_x}(\WF(\tau_0))}^{Zar}.$$
\end{cor}

In view of
Theorem \ref{thm:cusp.dep.0},
this corollary, together with Theorem \ref{thm:cor} and Theorem \ref{thm:lus}, imply Theorem \ref{thm:main.exp}.
\bibliographystyle{alpha}
\bibliography{Ramibib}

\begin{thebibliography}{CMBOb}

\bibitem[AA07]{AAu}
Pramod~N. Achar and Anne-Marie Aubert.
\newblock Supports unipotents de faisceaux caract\`eres.
\newblock {\em J. Inst. Math. Jussieu}, 6(2):173--207, 2007.

\bibitem[BM97]{BM}
Dan Barbasch and Allen Moy.
\newblock Local character expansions.
\newblock {\em Ann. Sci. \'{E}cole Norm. Sup. (4)}, 30(5):553--567, 1997.

\bibitem[CMBOa]{CMBO}
Dan Ciubotaru, Lucas Mason-Brown, and Emile Okada.
\newblock Some unipotent {A}rthur packets for reductive p-adic groups {I}.
\newblock {\em arXiv:2112.14354}.

\bibitem[CMBOb]{CMBO2}
Dan Ciubotaru, Lucas Mason-Brown, and Emile Okada.
\newblock The wavefront sets of unipotent supercuspidal representations.
\newblock {\em arXiv:2206.08628v2}.

\bibitem[Deb02a]{DebHom}
Stephen Debacker.
\newblock Homogeneity results for invariant distributions of a reductive
  {$p$}-adic group.
\newblock {\em Ann. Sci. \'{E}cole Norm. Sup. (4)}, 35(3):391--422, 2002.

\bibitem[DeB02b]{DebPar}
Stephen DeBacker.
\newblock Parametrizing nilpotent orbits via {B}ruhat-{T}its theory.
\newblock {\em Ann. of Math. (2)}, 156(1):295--332, 2002.

\bibitem[Kaw87]{Kaw}
Noriaki Kawanaka.
\newblock {Shintani lifting and {G}elfand-{G}raev representations}.
\newblock In {\em {The {A}rcata {C}onference on {R}epresentations of {F}inite
  {G}roups ({A}rcata, {C}alif., 1986)}}, volume~47 of {\em Proc. Sympos. Pure
  Math.}, pages 147--163. Amer. Math. Soc., Providence, RI, 1987.

\bibitem[Lat17]{Lath}
Peter Latham.
\newblock The unicity of types for depth-zero supercuspidal representations.
\newblock {\em Represent. Theory}, 21:590--610, 2017.

\bibitem[Lus85]{Lus}
George Lusztig.
\newblock Character sheaves. {I}.
\newblock {\em Adv. in Math.}, 56(3):193--237, 1985.

\bibitem[MP94]{MP94}
Allen Moy and Gopal Prasad.
\newblock Unrefined minimal {$K$}-types for {$p$}-adic groups.
\newblock {\em Invent. Math.}, 116(1-3):393--408, 1994.

\bibitem[MP96]{MP96}
Allen Moy and Gopal Prasad.
\newblock Jacquet functors and unrefined minimal {$K$}-types.
\newblock {\em Comment. Math. Helv.}, 71(1):98--121, 1996.

\bibitem[MW87]{MW87}
C.~M{\oe}glin and J.-L. Waldspurger.
\newblock {Mod\`eles de {W}hittaker d\'{e}g\'{e}n\'{e}r\'{e}s pour des groupes
  {$p$}-adiques}.
\newblock {\em Math. Z.}, 196(3):427--452, 1987.

\bibitem[Oka]{Oka}
E.~Okada.
\newblock The wavefront set of spherical arthur representations.
\newblock {\em arXiv:2107.10591v2}.

\bibitem[Rab]{Rab}
Joseph Rabinoff.
\newblock The bruhat-tits building of a p-adic chevalley group and an
  application to representation theory.
\newblock {\em Preprint:
  \url{https://services.math.duke.edu/~jdr/papers/building.pdf}}.

\bibitem[Tay16]{Tay}
Jay Taylor.
\newblock Generalized {G}elfand-{G}raev representations in small
  characteristics.
\newblock {\em Nagoya Math. J.}, 224(1):93--167, 2016.

\bibitem[Tsa]{Tsai}
Cheng-Chang Tsai.
\newblock Geometric wave-front set may not be a singleton.
\newblock {\em J. Amer. Math. Soc. DOI:10.1090/jams/1031}.

\bibitem[Wal18]{Wal18}
Jean-Loup Waldspurger.
\newblock Repr\'{e}sentations de r\'{e}duction unipotente pour {${\rm
  SO}(2n+1)$}, {III}: exemples de fronts d'onde.
\newblock {\em Algebra Number Theory}, 12(5):1107--1171, 2018.

\bibitem[Wal20]{Wal20}
Jean-Loup Waldspurger.
\newblock Fronts d'onde des repr\'{e}sentations temp\'{e}r\'{e}es et de
  r\'{e}duction unipotente pour {${\rm SO}(2n+1)$}.
\newblock {\em Tunis. J. Math.}, 2(1):43--95, 2020.

\end{thebibliography}

\end{document}